\documentclass{amsart}
\usepackage[english]{babel}
\usepackage{amsmath,latexsym,amssymb,verbatim}
\usepackage{pdfsync}

\usepackage[all]{xy}

\usepackage{hyperref}

\newtheorem{theorem}{Theorem}[section]
\newtheorem*{theorem-non}{Theorem}
\newtheorem{lemma}[theorem]{Lemma}
\newtheorem{proposition}[theorem]{Proposition}
\newtheorem{corollary}[theorem]{Corollary}
\newtheorem{remark}[theorem]{Remark}
\newtheorem{note}[theorem]{Note}

\newtheorem{Formula of adjoint functors}[theorem]{Formula of adjoint functors}

\newtheorem{example}[theorem]{Example}

\newtheorem{definition}[theorem]{Definition}
\newtheorem{notation}[theorem]{Notation}

\newtheorem{Adjunction formula}[theorem]{\indent\sc Adjunction formula}
\newtheorem{hypothesis}[theorem]{Hypothesis}

\DeclareMathOperator{\Id}{{Id}}
\DeclareMathOperator{\Hom}{{Hom}}

\DeclareMathOperator{\Ker}{{Ker}}
\DeclareMathOperator{\Ima}{{Im}}
\DeclareMathOperator{\ZZ}{{\mathbb Z}}
\DeclareMathOperator{\R}{{\mathcal R}}
\newcommand{\dosflechasa}[3][]{\xymatrix@1{\ar@<1ex>[r]^-{#2}
\ar@<-1ex>[r]_-{#3} & }}
\newcommand{\dosflechas}{{\xymatrix@1  {\ar@<1ex>[r]
\ar@<-1ex>[r] & }}}
\newcommand{\dosflechasab}[3][]{\xymatrix@1  @C10pt {\ar@<1ex>[r]^-{#2}
\ar@<-1ex>[r]_-{#3} & }}
\newcommand{\dosflechasb}{{\xymatrix@1 @C10pt {\ar@<1ex>[r]
\ar@<-1ex>[r] & }}}

\NoCompileMatrices

\input arrow.tex

\begin{document}

\title{Reflexivity of modules}

\author{Adri\'an Gordillo-Merino, Jos\'e Navarro and Pedro Sancho}
\address{Universidad de Extremadura. \newline All authors have been partially supported by Junta de Extremadura and FEDER funds.}


\begin{abstract} Let $\,R\,$ be an associative  ring with unit.  
We consider $\,R-$modules as functors in the following way: if $\,M\,$ is a (left) $R$-module, let $\,\mathcal M\,$ be the  functor of $\,\mathcal R-$modules defined by $\,\mathcal M(S) := S­\otimes_R M\,$ for every $\,R-$algebra $\,S$. 
With the corresponding notion of dual functor, we prove that the natural 
morphism of functors $\,\mathcal M\to \mathcal M^{**}\,$ is an isomorphism. 
\end{abstract}

\maketitle

\section{Introduction}

Every undergraduate student knows the following, elementary Linear Algebra fact: if $k\,$ is a field, $\,V\,$ is a $k$-vector space and $\,V^*=\Hom_k(V,k)\,$ is its dual vector space, then the natural morphism $$V\to V^{**},\, v\mapsto \tilde v\,,\,\,\,\mathrm{where}\,\,\, \tilde v(w):= w(v), \,\forall \, w\in V^*,$$
is not an isomorphism in general---only if  $\,V\,$ is finite-dimensional. 

In addition, if $\,R\,$ is a commutative ring and $M^*=\Hom_R(M,R)\,$ denotes the dual of  an $\,R$-module $\,M\,$, the natural morphism $\,M\to M^{**}\,$ may not be an isomorphism, even if $\,M\,$ is finitely generated---just take $\,R=\ZZ\,$ and $\,M=\ZZ/2\ZZ\,$, so that $\,M^*=\Hom_{\mathbb{Z}}(\ZZ/2\ZZ,\ZZ)=0\,$, and hence $\,M^{**}=0$.

However, it should be noticed that, if we consider modules as functors on the category of commutative $\,R$-algebras, and the linear dual is the corresponding functor of homomorphisms, then these module functors {\it are} reflexive (\cite{Amel}, \cite{gabriel}). The aim of this paper is to extend this result to modules defined over non-commutative rings.

\bigskip

To be more precise, let $\,R\,$ be an associative ring with unit, $M$ an $R$-module, and $N$ a right $R$-module. Consider the following covariant functors defined on the category of $R$-algebras:

The functor of rings $\,\mathcal{R}\,$, defined, for any  $R$-algebra $\,S\,$, as
$$\,\mathcal R(S):=S\,,$$ 
and  the functors of $\,\mathcal R$-modules $\mathcal M$ and $\mathcal N^*\,$, defined by 
$$\mathcal M(S):=S\otimes_R M\,,$$
$$\mathcal N^*(S):=\Hom_S(N\otimes_RS,S)\,,$$
for any $R$-algebra $S$. 

\medskip

Having adopted this point of view, this is our main result (Th. \ref{reflex}):

\begin{theorem-non} Let $M$ be an $R$-module. The natural morphism of $\mathcal R$-modules $$\mathcal M \to \mathcal M^{**}$$
is an isomorphism.
\end{theorem-non}

When $\,R\,$ is a commutative ring,  this theorem has been proved for finitely generated modules 
using the language of sheaves in the big Zariski topology, in  \cite{Hirschowitz}, and it is
implicit in \cite[II,\textsection 1,2.5]{gabriel}. The reflexivity of these quasi-coherent $\,\mathcal{R}$-modules $\,\mathcal{M}\,$ has been used for a variety of applications in theory of linear representations of affine group schemes (\cite{Amel},\cite{Pedro1},\cite{Pedro2}).
Likewise, we think that this new reflexivity theorem will be useful in the theory of comodules over non-commutative rings.

\bigskip

Let us briefly sketch how we get to prove it. 

The first two sections include preliminary definitions, as well as certain technicalities that will be required later on. In particular, we observe that any $\,R-$module $\,M\,$ can be described as a kernel of a morphism (of groups) between algebras:
\begin{align*}
R\langle M\rangle &\xrightarrow{\quad q_1 - q_2 \quad} R\langle M\oplus Rx\rangle \\
p(m) & \quad \longmapsto \quad p(m\cdot x)- p(m)\cdot x
\end{align*} where $\,R\langle M\rangle\,$ stands for the $\,R-$algebra generated by $\,M\,$ (see \ref{N3.4}).


Using this idea, Section \ref{SectExt} is devoted to proving that every $\mathcal{R}-$module $\mathbb{F}\,$ naturally extends to a functor $\overline{\mathbb{F}}\,$ from the category of right $R-$modules to the category of abelian groups.

As examples, for any $\,R-$module $\,M\,$, the extension of $\,\mathcal{M}\,$ is
$$\aligned \overline{\mathcal M}(Q)=\Ker [Q\otimes_R M\otimes_{\ZZ} R&\, \overset{p_1-p_2}\longrightarrow \,Q\otimes_R M\otimes_{\ZZ} R\otimes_{\ZZ} R]\\  q\otimes m\otimes r & \,\longmapsto \,q\otimes m\otimes r\otimes 1-q\otimes m\otimes 1\otimes r\endaligned$$ 
and, for any right $\,R-$module $\,N\,$, the extension of $\,\mathcal{N}^*\,$ is
$$\overline{\mathcal N^*}(Q)=\Hom_R(N,Q),$$
which is the the functor of (co)points of the $\,R-$module $\,N$.

Under some assumptions, we are able to prove the existence of the following isomorphism (Theorem \ref{L5.111}):
$$\Hom_{\R}(\mathbb F,\mathbb F')=\Hom_{\R}(\overline {\mathbb F},\overline{\mathbb F'})\,.$$

Finally, in Section \ref{lafinal} we use the isomorphism just mentioned and Yoneda's lemma to prove that
$${\Hom}_{\mathcal R} ({\mathcal N^*}, {\mathcal M})=\overline{\mathcal M}(N) \, .$$
and, as a corollary, that $\,\mathcal M=\mathcal M^{**}$.


\medskip

An effort has been made to make this paper as self-contained as possible.

\section{Preliminaries}\label{preliminar}

Let $\,R\,$ be an associative ring with  unit, and let $\,\mathcal R\,$ be the covariant functor from the category of $\,R$-algebras, $\,R\text{-Alg}\,$, to the category of rings, defined by $\,{\mathcal R}(S):=S$, for any $R$-algebra $\,S$.


\medskip
\begin{definition}
A {\sl functor of $\,\mathcal{R}$-modules} is a covariant functor  $\,\mathbb M \colon R\text{-Alg} \to \text{Ab} \,$ together with a morphism of functors of sets $\,{\mathcal R}\times \mathbb M\to \mathbb M\,$ that endows
$\,\mathbb M(S)\,$ with an $\,S$-module structure, for any $\,R$-algebra $\,S$. 

A {\sl morphism of $\,{\mathcal R}$-modules} $\,f\colon \mathbb M\to \mathbb M'\,$
is a morphism of functors such that the morphisms $\,f_{S}\colon \mathbb M({S})\to
\mathbb M'({S})\,$ are morphisms of $\,{S}$-modules.
\end{definition}
\medskip

If $\,S\,$ is an $R$-algebra, the restriction of an $\,{\mathcal R}$-module $\,\mathbb M\,$ to the category of ${S}$-algebras will be written $$\mathbb M_{\mid {S}}(S'):=\mathbb M(S'),$$ for any commutative ${S}$-algebra $S'$.

\medskip
\begin{definition}
The {\sl functor of homomorphisms} $\,{\mathbb Hom}_{{\mathcal R}}(\mathbb M,\mathbb M')\,$ is the covariant functor $\, R\text{-Alg} \to \text{Ab}\,$ defined by $${\mathbb Hom}_{{\mathcal R}}(\mathbb M,\mathbb M')({S}):={\rm Hom}_{\mathcal {S}}(\mathbb M_{|{S}}, \mathbb M'_{|{S}}), $$ where $\,\Hom_{{\mathcal S}}(\mathbb M_{|{S}},\mathbb M'_{|{S}})$ stands for the set\footnote{In this paper, we will only consider well-defined functors $\,{\mathbb Hom}_{{\mathcal R}}(\mathbb M,\mathbb M')$, that is to say, functors such that $\,\Hom_{\mathcal {S}}(\mathbb M_{|{S}},\mathbb {M'}_{|{S}})\,$ is a set, for any $R$-algebra ${S}$.} of all morphisms of $\,{\mathcal S}$-modules from $\,\mathbb M_{|{S}}\,$ to $\,\mathbb M'_{|{S'}}$.

In particular, the {\sl dual} of an $\mathcal{R}$-module $\,\mathbb{M}\,$ is the functor $$\,\mathbb M^*:=\mathbb Hom_{{\mathcal R}}(\mathbb M,\mathcal R). $$
\end{definition}
\medskip

In the following, it will also be convenient to consider another notion of dual module:
\begin{definition}  If $\,\mathbb M\,$ is an $\,\mathcal R$-module, the {\sl extended dual} $\,\mathbb M^\vee\,$ is the following functor $\, R\text{-Alg} \to \text{Ab}\,$ 
$$\mathbb M^\vee (S):=
\Hom_{\mathcal R}(\mathbb M,\mathcal S) \ . $$\end{definition}



\medskip
\begin{definition} The {\sl quasicoherent $\mathcal{R}$-module} associated to an $R$-module $\,M\,$ is the following covariant functor 
$$\,\mathcal{M} \colon R{\text-Alg} \to \text{Ab} \, , \quad \mathcal{M} (S) := S \otimes_R M. $$
\end{definition}
\medskip


Quasi-coherent modules are determined by its global sections. In particular, we will make use of the following statement, whose proof is immediate:

\medskip
\begin{proposition} \label{tercer}
Restriction to global sections $\,f\mapsto f_R\,$ defines a bijection:
$${\rm Hom}_{\mathcal R} ({\mathcal M}, \mathbb M) = {\rm Hom}_R (M, \mathbb M(R))\,, $$ for any quasicoherent $\,\mathcal{R}$-module $\,\mathcal{M}\,$ and any $\,{\mathcal R}$-module $\,\mathbb M$.
\end{proposition}
\medskip

As a consequence, both notions of dual module introduced above coincide on quasi-coherent modules; that is, $\,\mathcal M^*=\mathcal M^\vee\,$. 

In fact, if $\,S\,$ is an $R$-algebra, then 
$$\,\mathcal{M}^\vee (S) = \Hom_{\mathcal R}(\mathcal{M},\mathcal S) = \Hom_{R}(M , S)\,$$
and, as $\,{\mathcal M}_{\mid {S}}\,$ is the quasi-coherent $\mathcal {S}$-module associated to $\,S\otimes_R M \,$, 
$$\mathcal M^*(S)=\Hom_{\mathcal S}(\mathcal M_{|S},\mathcal S)=\Hom_S(S\otimes_R M,S)=\Hom_R(M,S) = \mathcal{M}^\vee (S) \ . $$




\bigskip
Finally, any definition or statement in the category of $\,\mathcal R$-modules has a corresponding definition or statement in the category of right $\,\mathcal R$-modules, that we will use without more explicit mention.

As examples, if $\,\mathbb{M}\,$ is an $\,\mathcal{R}$-module, then $\,\mathbb M^* = {\mathbb Hom}_{\mathcal R}(\mathbb M,\mathcal R)\,$ is a right $\,\mathcal R$-module. If $\,\mathbb N\,$ is a right $\,\mathcal R$-module, then the dual module defined by 
$$\mathbb N^*:=\mathbb Hom_{\mathcal R}(\mathbb N,\mathcal R)$$ is an $\,\mathcal R$-module, etc.

\section{Modules as kernels of morphisms between algebras}\label{hypo}

\begin{hypothesis} \label{HP2}
Let $N$ be a right  $R$-module and let $M$ be an $R$-module.  The sequence of morphisms of groups
$$N\otimes_RM\overset{i}\to N\otimes_R M\otimes_{\ZZ} R\dosflechasa{p_1}{p_2}   N\otimes_R M\otimes_{\ZZ} R\otimes_{\ZZ} R$$
where $\,i(n\otimes m):=n\otimes m\otimes 1$, $\,p_1(n\otimes m\otimes r):=n\otimes m\otimes r\otimes 1$ and
$\,p_2(n\otimes m\otimes r):=n\otimes m\otimes 1\otimes r$, is exact.
\end{hypothesis}

\medskip
The following three Propositions provide situations where this hypothesis is satisfied.

\begin{proposition} \label{super} Let $\,N\,$ be a right $\,R$-module and let $\,M\,$ be an $\,R$-module. If  $\,M\,$ (or $\,N\,$) is an $\,R$-bimodule or a flat module, then Hypothesis \ref{HP2} is satisfied.
\end{proposition}

\begin{proof} Suppose that $\,M\,$ is a bimodule.
It is clear that $\,\Ima i\subseteq \Ker(p_1-p_2)$. 
Let 
$\,s\colon N\otimes_R M\otimes_{\ZZ} R\to N\otimes_R M $, $s(n\otimes m\otimes r)=n\otimes mr$
and $\,s'\colon N\otimes_R M\otimes_{\ZZ} R \otimes_{\ZZ} R \to N\otimes_R M \otimes_{\ZZ} R$, $\,s'(n\otimes m\otimes r\otimes r')=n\otimes mr\otimes r'$. Observe that $\,s\circ i=\Id$, so that $i$ is injective. Also, $\,s'\circ p_2=\Id\,$ and $\,s'\circ p_1=i\circ s$. Thus, if $\,x\in \Ker(p_1-p_2)$, then $\,p_1(x)-p_2(x)=0$; hence, $\,0=s'(p_1(x))-s'(p_2(x))=i(s(x))-x\,$ and $\,x\in \Ima i$.

In particular, taking the bimodule $\,M=R\,$, the following sequence of morphisms of groups is exact:
$$N\overset{i}\to N\otimes_{\ZZ} R\dosflechasa{p_1}{p_2}   N\otimes_{\ZZ} R\otimes_{\ZZ} R \ . $$ 
Thus, if $\,M\,$ is flat, tensoring by $\,M\,$ it  also follows that Hypothesis \ref{HP2} is satisfied. 

\end{proof}

\begin{proposition} \label{super2} Hypothesis \ref{HP2} is satisfied if there exists a central subalgebra $\,R'\subseteq R\,$ such that $\,Q\to Q\otimes_{R'} R\,$ is injective, for any $\,R'$-module $\,Q\,$.
\end{proposition}

\begin{proof}
Let us write $\,M':=M\otimes_{R'} R\,$, which is a bimodule as follows:
$r_1\cdot (m\otimes  r)\cdot r_2=r_1m\otimes rr_2$. The morphism of $R$-modules $i\colon M\to M'$, $i(m):= m\otimes 1$ is universally injective: Given an $R$-module $P$, put $Q:=P\otimes_R M$. Then, the morphism $P\otimes_R M=Q\to Q\otimes_{R'} R=P\otimes_R M'$ is injective.

Put  $Q:=M'/M$ and $M'':=Q\otimes_{R'} R$.  Let $p_1$ be the composite morphism
 $M'\to M'/M=Q\to Q\otimes_{R'}R =M''$.
The sequence of morphisms of $R$-modules
$$0\to M\overset i\to M'\overset p\to M''$$
is universally exact. Consider the following commutative diagram
$$\xymatrix @R10pt @C10pt {0\ar[r] & N\otimes_R M \ar[r]^-{Id\otimes i} \ar[d] & N\otimes_RM' \ar[r]^-{Id\otimes p}  \ar[d] & N\otimes_RM''\ar[d] \\ 0\ar[r] &  N\otimes_R M \otimes_{\ZZ} R\ar[r]^-{Id\otimes i\otimes Id}  \ar@<1ex>[d] \ar@<-1ex>[d]  & N\otimes_R M' \otimes_{\ZZ} R\ar[r]^-{Id\otimes p\otimes Id}  \ar@<1ex>[d] \ar@<-1ex>[d] & N\otimes_RM''\otimes_{\ZZ} R \ar@<1ex>[d] \ar@<-1ex>[d] \\ 0\ar[r] & 
 N\otimes_R M \otimes_{\ZZ} R \otimes_{\ZZ} R\ar[r]^-{i'} & N\otimes_RM'\otimes_{\ZZ} R\otimes_{\ZZ} R  \ar[r]^-{p'}  &N\otimes_R M''\otimes_{\ZZ} R\otimes_{\ZZ} R
}$$ (where $i'=\Id\otimes i\otimes Id\otimes Id$ and $p'=\Id\otimes p\otimes Id\otimes Id$) whose rows are exact, as well as both the second and third columns, by Proposition \ref{super}. Hence, the first column is exact too.

\end{proof}

\begin{notation} \label{N3.4}
If $\,M\,$ is an $\,R$-module, observe that $\,M\otimes_{\ZZ} R\,$ is an $R$-bimodule and we can 
consider the tensorial $\,R$-algebra 
$$R\langle M\rangle :=T^\cdot_{R} (M\otimes_{\ZZ} R)=(T^{\cdot}_{\ZZ} M)\otimes_{\ZZ} R \, . $$\end{notation}
\medskip

\begin{remark} If $N$ is a right $R$-module, then:
$$R\langle N\rangle:=T^\cdot_R(R\otimes_{\ZZ} N) \, . $$ 
\end{remark}
\medskip

\begin{lemma} The following functorial map is bijective:
$$\Hom_{R-alg}(R\langle M\rangle, S)\to \Hom_{R}(M,S) \ , \quad f\mapsto f'\  ,$$
where $\,f'(m):=f(m\otimes 1)\,$ for any $\, m\in M\,$.\end{lemma}

\begin{proof} $ \Hom_{R-alg}( T^\cdot_{R} (M\otimes_{\ZZ} R),S)=\Hom_{R\otimes_{\ZZ} R} (M\otimes_{\ZZ} R,S) =\Hom_R(M,S).$
\end{proof} 
\medskip 
 
 Any $\,R$-linear morphism $\,\phi\colon M\to M'\,$ uniquely extends to a morphism of $\,R$-algebras $\,\tilde \phi\colon R\langle M\rangle \to R\langle M'\rangle$, $\,m\otimes 1\mapsto \phi(m)\otimes 1$.

If we use the notation $M\overset {n}\cdots M\cdot R:=M\otimes_{\ZZ}\overset {n}\cdots\otimes_{\ZZ} M\otimes_{\ZZ} R$, $m_1\cdots m_{n}\cdot r\mapsto m_1\otimes\cdots\otimes m_n\otimes r$, then
$$R\langle M\rangle =\oplus_{n=0}^\infty \,M\overset {n}\cdots M\cdot R\,,$$ and the product in this algebra can be written as follows:
$$(m_1\cdots m_{n}\cdot r)\cdot (m'_1\cdots m'_{n'}\cdot r')=
m_1\cdots m_{n}\cdot (r m'_1)\cdot m'_2\cdots m'_{n'}\cdot r'.$$

\medskip

\medskip
\begin{notation} Let us use the following notation 
$$M\oplus Rx:=M\oplus R \quad , \quad (m,r\cdot x)\mapsto (m,r) \, . $$
\end{notation}
\medskip


\begin{lemma} \label{reperab} Let $\,M\,$ be an $\,R$-module and $\,N\,$ a right $\,R$-module. Then, 
$$\Ker[N\otimes_R R\langle M\rangle\dosflechasab{q_1}{q_2} N\otimes_RR\langle M\oplus Rx\rangle]=\Ker[N\otimes_R M\otimes_{\ZZ} R\dosflechasab{p_1}{p_2} N\otimes_R   M\otimes_{\ZZ} R\otimes_{\ZZ} R],$$
where for any given $\,p(m)=\sum m_{i_1}\cdots m_{i_s}\cdot r_{i_1\ldots i_s}\in R\langle M\rangle\,$ and $\,n\in N\,$ the maps $\,q_1\,$ and $\,q_2\,$ are defined as follows:
$$q_1(n\otimes p(m)):=n\otimes p(mx):=n\otimes \sum m_{i_1}x\cdots m_{i_s}x\cdot r_{i_1\ldots i_s}$$ and $$q_2(n\otimes p(m)):=n\otimes p(m)x:=n\otimes \sum m_{i_1}\cdots m_{i_s}\cdot r_{i_1\ldots i_s}\cdot x \ . $$  
\end{lemma}

\begin{proof}  It is easy to prove that the kernel of the morphism
$$N\otimes_R R\langle M\rangle\to N\otimes_R  R\langle M\rangle[x] , \,\,n\otimes p(m)\mapsto
n\otimes (p(m)x-p(mx))$$
is included in  $N\otimes_R M\otimes_{\ZZ} R$. 
Observe that the morphism of $R$-algebras $R\langle M\oplus Rx\rangle\to R\langle M\rangle[x]$, $m\mapsto m$ and $x\mapsto x$, is an epimorphism. 

Then,  $$\Ker(q_1-q_2)\subseteq  N\otimes_R M\otimes_{\ZZ} R$$ and 
$\Ker(q_1-q_2)=\Ker(p_1-p_2).$
\end{proof}

As a consequence of this Lemma, it readily follows:

\begin{proposition} \label{repera} Hypothesis \ref{HP2} is satisfied if and only if 
the following sequence of morphisms of groups is exact:
$$\xymatrix @R6pt {N\otimes_R M\ar[r] &  N\otimes_RR\langle M\rangle  \ar@<1ex>[r]^-{q_1}
\ar@<-1ex>[r]_-{q_2} &  
N\otimes_RR\langle M\oplus Rx\rangle
\\  n\otimes m \ar@{|->}[r] &  n\otimes m,\,\,n\otimes p(m) \ar@{|->}[r]<1ex>
\ar@{|->}[r]<-1ex>&
n\otimes p(mx),\,n\otimes p(m)x\,.} $$
\end{proposition} 
\medskip

\section{Extension of a functor on the category of algebras to a functor on the category of modules}\label{SectExt}

Let $\,\mathbb F\,$ be a functor defined on the category of $\,R$-algebras. Our aim in this Section is to define its extension to a functor $\, \overline{\mathbb F}\,$ defined on the category of $\,R$-modules. Using this procedure, the reflexivity theorem will be \textsl{recovered} as a particular case of Yoneda's lemma.

\begin{notation} Let $M$ be an $R-$module. Consider the morphism of $R-$algebras $$h_x: R\langle M\oplus Rx\rangle \longrightarrow R\langle M\oplus Rx\rangle $$ defined by $h_x(m)=m\cdot x\,$, for any $m\in M\,$, and $h_x(x)=x\,.$ 
\end{notation}

\medskip
\begin{definition} \label{D4.2} The {\sl extension} of a functor $\,\mathbb F\,$ of right $\,\R -$modules is the functor $\,\overline{\mathbb F}\,,$ from the category of $\,R-$modules to the  category of abelian groups, defined by $$\overline{\mathbb F}(M):=\Ker
[\mathbb  F(R\langle M\rangle)\to \mathbb  F(R\langle M\oplus Rx\rangle), \, f\mapsto \mathbb F(h_x)(f)- f\cdot x ],$$ 
for any $R$-module $M$ and any $f\in  \mathbb F(R\langle M\rangle)$.
\end{definition}

If $\,w\colon M\to M'\,$ is a morphism of $R$-modules, it induces morphisms of $R$-algebras
$$\tilde w\colon R\langle M\rangle 
\to R\langle M'\rangle \ ,  \quad \tilde w(m)=w(m)$$ and 
$\tilde{\tilde w}\colon R\langle M\oplus Rx\rangle 
\to R\langle M'\oplus Rx\rangle$, $\tilde{\tilde w}(m)=w(m)$, $\tilde{\tilde w}(x)=x$.
Observe that $\tilde{\tilde w}\circ h_x=h_x\circ \tilde{\tilde w}$. Hence,
we have the morphism 
$$\overline{\mathbb F}(w)\colon \overline{\mathbb F}(M)\to \overline{\mathbb F}(M')\, , \quad \overline{\mathbb F}(w)(f):=\mathbb F(\tilde w)(f)$$ for  any $f\in\overline{\mathbb F}(M)\subset \mathbb F(R\langle M\rangle)$.

\begin{note} In a similar vein, we can define the {\sl extension} of a functor $\,\mathbb F\,$ of $\R$-modules, which is a functor $\,\overline{\mathbb F}\,$ from the category of right $R$-modules to the category of abelian groups.\end{note}

\begin{proposition}\label{repera2} Hypothesis \ref{HP2} is satisfied if and only if 
$$
\overline{\mathcal N}(M)=N\otimes_R M \ .
$$
\end{proposition}

\begin{proof} It is an immediate consequence of Lemma \ref{reperab}.
\end{proof}

\medskip
\begin{remark} \label{repera2b} Observe that 
$$\overline{\mathcal N}(M)=\Ker[N\otimes_R M\otimes_{\ZZ} R\overset{p_1-p_2}\longrightarrow N\otimes_R M\otimes_{\ZZ} R\otimes_{\ZZ} R]=\overline{\mathcal M}(N)\ . $$
\end{remark}
\medskip

\medskip
\begin{proposition}\label{1.30}   Let $\mathbb F$ be a functor of $\mathcal R$-modules. Then, $$\overline{\mathbb F^\vee}(M)=\Hom_{\mathcal R}(\mathbb F,\mathcal M).$$

\end{proposition}

\begin{proof} By Propositions \ref{repera} and \ref{super}, the sequence of morphisms 
$$\xymatrix @R6pt {0 \ar[r] & M\ar[r] &  R\langle M\rangle \ar@<1ex>[r]
\ar@<-1ex>[r]&  R\langle M\oplus Rx\rangle\\ & m \ar@{|->}[r] & m,\quad p(m) \ar@{|->}[r]<1ex>
\ar@{|->}[r]<-1ex>&
p(mx),\,p(m)x}$$
remains exact when tensoring by $\,R$-algebras. Hence,
$\overline{\mathbb F^\vee}(M)=\Hom_{\mathcal R}(\mathbb F,\mathcal M).$

\end{proof}
\medskip

Let $S$ be an $R$-algebra and $s\in S$. The morphism of $R$-modules $\cdot s\colon  S\to S$, $s'\mapsto s'\cdot s$ induces the morphism of 
$R$-algebras $\tilde{\cdot s}\colon R\langle S\rangle\to R\langle S\rangle$, $s'\mapsto s'\cdot s$, which in turn induces the morphism of
groups
$$\overline {\mathbb F}(\cdot s)\colon \overline {\mathbb F}(S)\to \overline {\mathbb F}(S)\ , \quad f\mapsto \mathbb F(\tilde{\cdot s})(f) \ .$$ 

\medskip
\begin{proposition} Let $S$ be an $R$-algebra and $s\in S$. Then, for any $\,f\in \overline{\mathbb F}(S) \subset \mathbb F(R\langle S\rangle)$
$$\overline {\mathbb F}(\cdot s)(f)=f\cdot s . $$

Therefore, $\,\overline{\mathbb F}(S)\,$ has a natural structure of right $\,S$-module
\end{proposition}

\begin{proof} Given $f\in\overline{\mathbb F}(S)$, we know that $\mathbb F(h_x)(f)=f\cdot x$ in $\mathbb F( R\langle S\oplus Rx\rangle) $. Consider the morphism of $R$-algebras $R\langle S\oplus Rx\rangle \overset{x=s}\to R\langle S\rangle$, $s'\mapsto s'$ and $x\mapsto s$. We have the commutative diagrams
$$\xymatrix{R\langle S\oplus Rx\rangle \ar[r]^-{x=1} \ar[d]^-{h_x} & R\langle S\rangle \ar[d]^-{\tilde{\cdot s} }\\ \R\langle S\oplus Rx\rangle \ar[r]^-{x=s}
& R\langle S\rangle}\quad 
\xymatrix{\mathbb F(R\langle S\oplus Rx\rangle) \ar[r]^-{\mathbb F(x=1)} \ar[d]^-{\mathbb F(h_x)} & \mathbb F(R\langle S\rangle) \ar[d]^-{\mathbb F(\tilde{\cdot s}) }\\ \mathbb F(R\langle S\oplus Rx\rangle) \ar[r]^-{\mathbb F(x=s)}
& \mathbb F(R\langle S\rangle) }
$$
and the composite morphism $\mathbb F(R\langle S\rangle)\to 
\mathbb F(R\langle S\oplus Rx\rangle)\overset{\mathbb F(x=s)}\to \mathbb F(R\langle S\rangle) $ is the identity, for any $s\in S$.
Hence, $\mathbb F(x=s)(f)=f$ and 
$$\aligned \overline{\mathbb F}(\cdot s)(f) & =\mathbb F(\tilde{\cdot s})(f)=
\mathbb F(\tilde{\cdot s})(\mathbb F(x=1)(f))=\mathbb F(x=s)(\mathbb F(h_x)(f))=
\mathbb F(x=s)(f\cdot x)\\ & =\mathbb F(x=s)(f)\cdot s=f\cdot s.\endaligned$$
\end{proof}

\medskip

Let $\,S\,$ be an $\,R$-algebra and let $\,\pi_S\colon R\langle S\rangle \to S\,$ be the morphism of $\,R$-algebras $\,s\mapsto s\,$, for any $\,s\in S$. Consider the following composition of morphisms of $S$-modules 
$$\overline{\mathbb F}(S)\subseteq \mathbb F(R\langle S\rangle)\overset{\mathbb F(\pi_S)}\longrightarrow \mathbb F(S)\,.$$

\begin{theorem} \label{L5.111} Let $\,\mathbb F\,$ be a right $\,\mathcal R$-module such that the natural morphism $$\,\mathbb F(\pi_S)_{|\overline{\mathbb F}(S)}\colon \overline{\mathbb F}(S)\to \mathbb F(S)\,$$ is an isomorphism, for any $\,R$-algebra $\,S\,$.

If $\,\mathbb F'\,$ is another right $\,\mathcal{R}$-module, there exists a natural isomorphism
$$\Hom_{\mathcal R}(\mathbb F,\mathbb F')\xrightarrow{\ \sim \ } \Hom_{\mathcal R}(\overline{\mathbb F} , \overline{\mathbb F'})  \,,$$
where $\Hom_{\mathcal R}(\overline{\mathbb F} , \overline{\mathbb F'})\,$ stands for the set of morphisms of functors of groups from $\overline{\mathbb{F}}\,$ to $\overline{\mathbb{F}'}\,$.
\end{theorem}

\begin{proof} First of all, any morphism of right $\,\R$-modules $\,\phi\colon \mathbb F\to \mathbb F'\,$ can be extended to a morphism of functors of groups $$\overline \phi\colon \overline{\mathbb F}\to \overline{\mathbb F}' \ , \quad \,\overline\phi_M(f):=\phi_{R\langle M\rangle}(f)\,,$$ for any $\,R$-module $\,M\,$ and any $\,f\in \overline{\mathbb F}(M)\subset {\mathbb F}(R\langle M\rangle)$.

On the other hand, given $\varphi\in \Hom_{\R}(\overline{\mathbb F},\overline{\mathbb F'})$, let $\tilde\varphi\in \Hom_{\mathcal R}(\mathbb F,\mathbb F')$ be defined by $$\tilde\varphi_S:=(\mathbb F'(\pi_S)\circ i'_S)\circ \varphi_S\circ (\mathbb F(\pi_S)\circ i_S)^{-1},$$ for any $R$-algebra $S$ ($i_S\colon \overline{\mathbb F}(S)\subseteq \mathbb F(R\langle S\rangle)$ and $i'_S\colon \overline{\mathbb F'}(S)\subseteq \mathbb F'(R\langle S\rangle)$ are the inclusion morphisms).

1. $\tilde{\overline \phi}=\phi$: The diagram
$$\xymatrix{\overline{\mathbb F}(S) \ar[d]^-{\overline \phi_S} \ar[r]^-{i_S} & \mathbb F(R\langle S\rangle) \ar[r]^-{\mathbb F(\pi_S)} \ar[d]^-{\phi_{R\langle S\rangle}} & \mathbb F(S) \ar[d]^-{\phi_S}\\
\overline{\mathbb F'}(S) \ar[r]^-{i'_S} & \mathbb F'(R\langle S\rangle) \ar[r]^-{\mathbb F'(\pi_S)} &
 \mathbb F'(S) }$$
is commutative. Hence, $\tilde{\overline\phi}_S:=(\mathbb F'(\pi_S)\circ i'_S)\circ \overline\phi_S\circ (\mathbb F(\pi_S)\circ i_S)^{-1}=\phi_S\circ \mathbb F(\pi_S)\circ i_S\circ (\mathbb F(\pi_S)\circ i_S)^{-1}=\phi_S$.

2. $\overline{\tilde\varphi}=\varphi$: The diagram
$$\xymatrix{\overline{\mathbb F}(N) \ar[r] \ar[d]^-{\varphi_N} & \overline{\mathbb F}(R\langle N\rangle) \ar[r] \ar[d]^-{\varphi_{R\langle N\rangle}} & \mathbb F(R\langle R\langle N\rangle\rangle) \ar[r] & \mathbb F(R\langle N\rangle) \ar[d]^-{\tilde\varphi_{R\langle N\rangle}} 
\\ \overline{\mathbb F'}(N) \ar[r] & \overline{\mathbb F'}(R\langle N\rangle) \ar[r]  & \mathbb F'(R\langle R\langle N\rangle\rangle) \ar[r] & \mathbb F'(R\langle N\rangle) }$$
is commutative. Hence, $(\overline{\tilde\varphi})_N={\tilde{\varphi}_N\,}_{|\overline{\mathbb F}(N)}=\varphi_N$.
\end{proof}
\medskip

\begin{example}\label{EjemploVee} Any extended dual $\,\mathbb{F}^\vee\,$ satisfies the hypothesis of Theorem \ref{L5.111}: the composition $\,\overline{\mathbb F^\vee }(S)\subseteq {\mathbb F^\vee }(R\langle S\rangle) \to \mathbb F^\vee (S)\,$ is the identity morphism, as follows from Proposition \ref{1.30} and the fact that
the composition $\,S\to R\langle S\rangle \overset{\pi_S}\to S\,$ is the identity morphism.
\end{example}

\section{Reflexivity theorem}\label{lafinal}

Let $\,M\,$ be an $\,R-$module. The functor $\,\overline{M^\vee}\,$ is precisely the functor of (co)points of $\,M\,$ in the category of $\,R-$modules: if $\,Q\,$ is another $\,R-$module, in virtue of Proposition \ref{1.30}:
$$\overline{\mathcal{M}^\vee}(Q) = \Hom_{\mathcal{R}} (\mathcal{M} , \mathcal{Q}) = \Hom_{R}(M , Q) \ . $$

\begin{theorem}\label{prop4} Let $\,M\,$ be an $\,R$-module and $\,N\,$ be a right $\,R$-module. 
Then, $${\Hom}_{\mathcal R} ({\mathcal M^*}, {\mathcal N})=\overline{\mathcal N}(M) \, .$$
\end{theorem}

\begin{proof} The dual $\,\mathcal{M}^* = \mathcal{M}^\vee\,$ satisfies the hypothesis of Theorem \ref{L5.111} (see Example \ref{EjemploVee}), so that
$$ {\Hom}_{\mathcal R} ({\mathcal M^*}, {\mathcal N})  = {\Hom}_{\mathcal R} ({\mathcal M^\vee}, {\mathcal N}) = {\Hom}_{\R} ({\overline{\mathcal M^\vee}}, {\overline{\mathcal N}}) \ .
$$

As $\,\overline{\mathcal{M}^\vee}\,$ is a functor of points, the statement now follows applying Yoneda's lemma:
$$
{\Hom}_{\R} ({\overline{\mathcal M^\vee}}, {\overline{\mathcal N}}) = \overline {\mathcal N}(M)\, . 
$$
\end{proof}

As a consequence of this Theorem and Proposition \ref{repera2}, we obtain the following formula:

\begin{corollary}\label{FormulaHipotesis} Let $\,M\,$ be an $\,R$-module and $\,N\,$ be a right $\,R$-module.
The following equality of abelian groups holds
$${\Hom}_{\mathcal R} ({\mathcal M^*}, {\mathcal N})= N \otimes_R M  $$ if and only if Hypothesis \ref{HP2} is satisfied.
\end{corollary}

\medskip
\begin{theorem} \label{reflex}
Let $\,M\,$ be an $\,R$-module. 
The natural morphism of $\,\mathcal R$-modules $$\mathcal M \longrightarrow \mathcal M^{**}$$ is an isomorphism.
\end{theorem}

\begin{proof} On the one hand,
$\mathcal M^{**}(S)=\Hom_{\mathcal S}({\mathcal M^*}_{|S},\mathcal S)=
\Hom_{\mathcal S}({\mathcal M_{|S}}^*,\mathcal S).$ 

On the other, any $\,R$-algebra $\,S\,$ is a bimodule, so that Hypothesis \ref{HP2} is satisfied (Proposition \ref{super}) and we can apply Corollary \ref{FormulaHipotesis}. As $\,\mathcal M_{|S}\,$ is the quasicoherent module associated to $\,S\otimes_R M\,$, it follows:
$$ \Hom_{\mathcal S}({\mathcal M_{|S}}^*,\mathcal S) =S\otimes_S S\otimes_R M=S\otimes_R M=\mathcal M(S) \ . $$
\end{proof}
\medskip

Finally, let us show how the techniques we have developed also allow to prove a reflexivity theorem for the extended dual of quasicoherent modules:

\begin{theorem} 
Let $\,M\,$ be an $\,R$-module. The natural morphism of $\,\mathcal R$-modules $$\mathcal M \longrightarrow \mathcal M^{\vee\vee}$$ is an isomorphism.
\end{theorem}

\begin{proof} It is a consequence of $\,\mathcal{M}^\vee = \mathcal{M}^*\,$ and Corollary \ref{FormulaHipotesis}: 
$$\mathcal M^{\vee\vee}(S)=\Hom_{\mathcal R}(\mathcal M^\vee,\mathcal S)= \Hom_{\mathcal R}(\mathcal M^*,\mathcal S) = S\otimes_R M=\mathcal M(S) \ .$$ 
\end{proof}

\end{document}